\documentclass[12pt]{article}
\title{Abelian groups definable in $p$-adically closed fields}
\author{Will Johnson and  Ningyuan Yao}

\usepackage{amsmath, amssymb, amsthm}    	
\usepackage{fullpage} 	
\usepackage{amscd}
\usepackage{hyperref}
\usepackage[all]{xy}
\usepackage{centernot}
\usepackage{color,xcolor} 

\DeclareMathOperator*{\ind}{\raise0.2ex\hbox{\ooalign{\hidewidth$\vert$\hidewidth\cr\raise-0.9ex\hbox{$\smile$}}}}

\newcommand{\dimind}{\ind^{\dim}}
\newcommand{\ba}{\bar{a}}
\newcommand{\bb}{\bar{b}}
\newcommand{\bx}{\bar{x}}
\newcommand{\by}{\bar{y}}
\newcommand{\pCF}{p\mathrm{CF}}

\newcommand{\Aut}{\operatorname{Aut}}

\newcommand{\acl}{\operatorname{acl}}
\newcommand{\dcl}{\operatorname{dcl}}
\newcommand{\tp}{\operatorname{tp}}

\DeclareMathOperator{\stab}{stab}
\DeclareMathOperator{\st}{st}
\newcommand{\trdeg}{\operatorname{tr.deg}}

\newcommand{\dpr}{\operatorname{dp-rk}}

\newcommand{\Q}{\mathbb{Q}}
\newcommand{\Qp}{\mathbb{Q}_p}

\newtheorem{theorem}{Theorem}[section] 
\newtheorem{lemma}[theorem]{Lemma}

\newtheorem{corollary}[theorem]{Corollary}
\newtheorem{fact}[theorem]{Fact}

\newtheorem{conjecture}[theorem]{Conjecture}

\newtheorem{proposition}[theorem]{Proposition}
\newtheorem{proposition-eh}[theorem]{Proposition(?)}
\newtheorem*{theorem-star}{Theorem}
\newtheorem*{conjecture-star}{Conjecture}
\newtheorem*{lemma-star}{Lemma}

\theoremstyle{definition}
\newtheorem{definition}[theorem]{Definition}

\newtheorem{remark}[theorem]{Remark}

\newtheorem*{acknowledgment}{Acknowledgments}

\newcommand{\Qq}{\mathbb{Q}}
\newcommand{\eq}{\mathrm{eq}}

\newcommand{\Zz}{\mathbb{Z}}

\newcommand{\Mm}{\mathbb{M}}

\newcommand{\Pp}{\mathbb{P}}

\newcommand{\sq}{\subseteq}

\begin{document}
\maketitle

\begin{abstract}
Recall that a group $G$ has finitely satisfiable generics (\textit{fsg}) or definable $f$-generics (\textit{dfg}) if there is a global type $p$ on $G$ and a small model $M_0$ such that every left translate of $p$ is finitely satisfiable in $M_0$ or definable over $M_0$, respectively.
We show that any abelian group definable in a $p$-adically closed field is an extension of a definably compact \textit{fsg} definable group by a \textit{dfg} definable group.  We discuss an approach which might prove a similar statement for interpretable abelian groups.  In the case where $G$ is an abelian group definable in the standard model $\Qp$, we show that $G^0 = G^{00}$, and that $G$ is an open subgroup of an algebraic group, up to finite factors.  This latter result can be seen as a rough classification of abelian definable groups in $\Qp$.
\end{abstract}

\section{Introduction}
In this paper we study abelian groups definable in $p$-adically closed
fields.  Recall that a definable group $G$ has \emph{finitely
satisfiable generics} (\textit{fsg}) if there is a global type on $G$,
finitely satisfiable in a small model, with boundedly many left
translates.  Similarly, $G$ has \emph{definable f-generics}
(\textit{dfg}) if there is a definable global type on $G$ with
boundedly many left translates.
  The main theorem of this paper is the following decomposition
of abelian definable groups into \textit{dfg} and \textit{fsg} components:
\begin{theorem} \label{th1}
  Suppose that $M$ is a $p$-adically closed field and $G$ is an
  abelian group definable in $M$.  Then there is a short exact
  sequence of definable groups
  \begin{equation*}
    1 \to H \to G \to C \to 1
  \end{equation*}
  where $H$ has \textit{dfg} and $C$ is definably compact and has
  \textit{fsg}.
\end{theorem}
An analogous decomposition for definably amenable groups in o-minimal structures was proved by Conversano and Pillay \cite[Propositions~4.6--7]{cp} (see also \cite[Fact~1.18]{pillay-yao0}).  Pillay and Yao asked whether such a decomposition exists for any definably amenable group in a distal theory \cite[Question~1.19]{pillay-yao0}; Theorem~\ref{th1} can be seen as evidence towards a positive answer.

When $M = \Qp$, we obtain two useful consequences from Theorem~\ref{th1}:
\begin{theorem} \label{th2}
  Suppose that $G$ is an abelian definable group in $\Qp$.
  \begin{enumerate}
  \item $G^{00} = G^0$.
  \item There is a finite index definable subgroup $E \subseteq G$ and a finite subgroup $F \subseteq E$ such that $E/F$ is isomorphic to an open
    subgroup of an algebraic group.
  \end{enumerate}
\end{theorem}
This yields a loose ``classification'' of abelian definable groups in
$\Qp$---up to finite factors, they are exactly the open subgroups of
algebraic groups.

\begin{acknowledgment}
  The first author was supported by the National Natural Science
  Foundation of China (Grant No.\@ 12101131). The second auothor was supported by the National Social Fund of China (Grant No.\@ 20CZX050). Section~\ref{sec5} was partially based on joint work with Zhentao Zhang, who declined to be an author on this paper.  
\end{acknowledgment}

\subsection{Outline}
In Section~\ref{tools}, we review some tools needed in the proof.  In
Section~\ref{sec3} we prove the decomposition in Theorem~\ref{th1}.
In Section~\ref{qp-case} we obtain the consequences for
$\Qp$-definable groups listed in Theorem~\ref{th2}.  In
Section~\ref{sec5} we discuss our original strategy for
Theorem~\ref{th1}, which suggests a generalization of
Theorem~\ref{th1} to interpretable groups (Conjecture~\ref{conj}).

There are also two appendices.  Appendix~\ref{ict} proves a technical
statement about topological properties of ict patterns in
interpretable groups, needed in Lemma~\ref{hard}.
Appendix~\ref{exts2} is on \textit{dfg} in short exact sequences, and
generalizes some facts in Section~\ref{exts} beyond the context of
$\pCF$.

\subsection{Notation and conventions}
``Definable'' means ``definable with parameters.''
We write the monster model as $\Mm$. A ``type'' is a complete type, and a ``partial type'' is a partial type.  Tuples are finite by default.  We usually write tuples as $a, b, x, y$ rather than $\ba, \bb, \bx, \by$.  We distinguish between ``real'' elements or tuples (in $\Mm$) and ``imaginaries'' (in $\Mm^\eq$), and we distinguish between ``definable'' (in $\Mm$) and ``interpretable'' (in $\Mm^\eq$).  The exception is Appendix~\ref{exts2}, where we work in $\Mm^\eq$.  If $D$ is a definable set, then $\ulcorner D \urcorner$ denotes its code, a tuple in $\Mm^\eq$.  If $p$ is a definable type, then $\ulcorner p \urcorner$ denotes its code, an infinite tuple in $\Mm^\eq$.

Throughout, $\pCF$ means the complete theory of $\Qp$, and a ``$p$-adically closed field'' is a model of this theory, or equivalently, a field elementarily equivalent to $\Qp$.  We do not consider ``$p$-adically closed fields'' in the broader sense (fields elementarily equivalent to finite extensions of $\Qp$), though we strongly suspect that all the results generalize to these theories.
We write the language of $\pCF$ as $\mathcal{L}$.  The language $\mathcal{L}$ should be one-sorted; otherwise the choice of $\mathcal{L}$ is irrelevant.

\section{Tools} \label{tools}
In this section, we review a few tools that will be needed in the
proof of the main theorems.  In Section~\ref{exts} we show that
certain properties ($G^0 = G^{00}$, \textit{dfg}) behave well in short
exact sequences.  In Section~\ref{csec} we show that we can take
quotients by certain \textit{dfg} groups without leaving the definable
category.
\subsection{Extensions} \label{exts}
Recall that $G^{00}$ and $G^0$ exist for definable groups $G$ in NIP theories \cite[Proposition~6.1]{HPP}.
\begin{lemma}[Assuming NIP]
  Let $\pi : G \to X$ be a surjective homomorphism of definable groups.  Then
  $\pi(G^{00}) = X^{00}$.
\end{lemma}
\begin{proof}
  There is a surjection $G/G^{00} \to X/\pi(G^{00})$, so
  $X/\pi(G^{00})$ is bounded and $\pi(G^{00}) \supseteq X^{00}$.
  There is an bijection $G/\pi^{-1}(X^{00}) \to X/X^{00}$, so
  $G/\pi^{-1}(X^{00})$ is bounded and $G^{00} \subseteq
  \pi^{-1}(X^{00})$.  This implies $\pi(G^{00}) \subseteq X^{00}$.
\end{proof}
\begin{lemma}[Assuming NIP] \label{0-00-ext}
  Let $1 \to H \to G \overset{\pi}{\to} X \to 1$ be a short exact
  sequence of definable groups.  If $H^0 = H^{00}$ and $X^0 = X^{00}$,
  then $G^0 = G^{00}$.
\end{lemma}
\begin{proof}
  The fact that $H^0 = H^{00}$ and $X^0 = X^{00}$ means that
  $H/H^{00}$ and $X/X^{00}$ are profinite.  The short exact sequence
  \begin{equation*}
    1 \to H/(H \cap G^{00}) \to G/G^{00} \to X/X^{00} \to 1 \tag{$\ast$}
  \end{equation*}
  shows that $H/(H \cap G^{00})$ is bounded, and then ($\ast$)
  is continuous in the logic topology.  As $H/(H \cap G^{00})$ is
  bounded, it must be a quotient of $H/H^{00}$ which is profinite.  Therefore $H/(H \cap G^{00})$ is profinite.  In
  the category of compact Hausdorff groups, an extension of a
  profinite group by a profinite group is profinite.  Therefore
  $G/G^{00}$ is profinite, which implies $G^0 = G^{00}$.
\end{proof}
Recall that $\pCF$ has definable Skolem functions.
\begin{lemma} \label{dfg-ext}
  Suppose that $\Mm$ is a saturated model of $\pCF$.  Let
  \begin{equation*}
    1\longrightarrow
    A\overset{i}{\longrightarrow}B\overset{\pi}{\longrightarrow} C
     \longrightarrow 1
  \end{equation*}
  be a short exact sequence of definable groups.  Then $B$ has
  \textit{dfg} iff $A$ and $C$ do.
\end{lemma}
\begin{proof}
  We prove the following:
  \begin{enumerate}
  \item If $B$ has \textit{dfg}, then $C$ has \textit{dfg}.
  \item If $B$ has \textit{dfg}, then $A$ has \textit{dfg}.
  \item If $A$ and $C$ have \textit{dfg}, then $B$ has \textit{dfg}.
  \end{enumerate}
  By definable Skolem functions, there is a definable function $f : C
  \to B$ which is a set-theoretic section of $\pi$, in the sense that
  $\pi(f(c)) = c$ for $c \in C$.  Now we proceed with the proofs:
  \begin{enumerate}
  \item If $\tp(b/\Mm)$ is a definable f-generic type in $B$, then
    $\tp(\pi(b)/\Mm)$ is a definable f-generic type in $C$.
  \item The proof is nearly identical to \cite[Lemmas~2.24,
    2.25]{pillay-yao}.  In an elementary extension $\Mm' \succeq \Mm$,
    take $b_0 \in B(\Mm')$ realizing a definable f-generic type in
    $B$.  Write $b_0$ as $a_0 \cdot f(\pi(b_0))$ for some $a_0 \in
    A(\Mm')$.  Then $a_0 \in \dcl(\Mm b_0)$, so $\tp(a_0/\Mm)$ is
    definable.  We claim that $\tp(a_0/\Mm)$ has boundedly many left
    translates, and is therefore a definable f-generic type in $A$.
    Note that $A^{00} \subseteq B^{00}$ because $A/(A \cap B^{00})
    \cong AB^{00}/B^{00}$ is bounded.  If $\delta \in A^{00}(\Mm)$,
    then $\tp(\delta \cdot b_0 / \Mm) = \tp(b_0 / \Mm)$, and therefore
    \begin{equation*}
      \tp(\delta \cdot b_0 \cdot f(\pi(\delta \cdot b_0))^{-1} / \Mm)
      = \tp(b_0 \cdot f(\pi(b_0))^{-1}/\Mm) = \tp(a_0/\Mm).
    \end{equation*}
    But $\pi(\delta \cdot b_0) = \pi(b_0)$, and so
    \begin{equation*}
      \tp(\delta \cdot b_0 \cdot f(\pi(\delta \cdot b_0))^{-1} / \Mm)
      = \tp(\delta \cdot b_0 \cdot f(\pi(b_0))^{-1} / \Mm)
      = \tp(\delta \cdot a_0 / \Mm).
    \end{equation*}
    Therefore $\tp(a_0/\Mm)$ is invariant under left translation by
    any $\delta \in A^{00}$, and it has boundedly many left
    translates.
  \item 
  Let $p(x)\in S_A(\Mm)$ and $q(y)\in
  S_C(\Mm)$ be \textit{dfg} types of $A$ and $C$ respectively.  Let
  $M_0$ be a small model defining the section $f$, the short exact
  sequence, and all the left translates of $p$ and $q$.

  In some elementary extension $\Mm' \succeq \Mm$, take $c_0 \models q$ and $a_0
  \models p|\Mm c_0$.  Then $\tp(a_0,c_0/\Mm)$ is $M_0$-definable---it is
  the Morley product of $p$ and $q$.  Let $b_0 = f(c_0) \cdot a_0$.  Then
  $\tp(b_0/\Mm)$ is again $M_0$-definable.  We claim that every left
  translate of $\tp(b_0/\Mm)$ is $M_0$-definable.

  Fix some $\delta \in B(\Mm)$.  Let $b_1 = \delta \cdot b_0$.  Let
  $c_1 = \pi(\delta) \cdot c_0$.  Let $\delta' = f(c_1)^{-1} \cdot
  \delta \cdot f(c_0)$.  Note
  \begin{equation*}
    \pi(\delta') = \pi(f(c_1))^{-1} \cdot \pi(\delta) \cdot
    \pi(f(c_0)) = c_1^{-1} \cdot \pi(\delta) \cdot c_0 = 1,
  \end{equation*}
  so $\delta' \in A(\Mm')$.
  Let $a_1 = \delta' \cdot a_0$.  Then
  \begin{equation*}
    b_1 = \delta \cdot b_0 = \delta \cdot f(c_0) \cdot a_0 = f(c_1)
    \cdot \delta' \cdot a_0 = f(c_1) \cdot a_1.
  \end{equation*}
  Now $\tp(c_1/\Mm) = \tp(\pi(\delta) \cdot c_0 / \Mm)$ is a
  left-translate of the \textit{dfg} type $\tp(c_0/\Mm) = q$, and so
  $\tp(c_1/\Mm)$ is $M_0$-definable.  If $U$ is $\dcl(\Mm c_0) =
  \dcl(\Mm c_1)$, then $\tp(a_1 / U) = \tp(\delta' \cdot a_0 / U)$ is
  a left translate of the \textit{dfg} type $\tp(a_0 / U) = p | U$
  (because $\delta' \in U$).  Therefore $\tp(a_1 / U)$ is again
  $M_0$-definable.  As $b_1 = f(c_1) \cdot a_1$, we see that
  $\tp(\delta \cdot b_0 / \Mm) = \tp(b_1/\Mm)$ is $M_0$-definable for
  the same reason that $\tp(b_0/\Mm)$ is $M_0$-definable, essentially because $\tp(c_1/\Mm)$ and $\tp(a_1/\Mm c_1)$ are $M_0$-definable. \qedhere
  \end{enumerate}
\end{proof}
\noindent See Theorem~\ref{dfg-ext-2} in the appendix for an alternate proof of (3) not using definable
Skolem functions.

\subsection{Codes and quotients} \label{csec}
Let $G$ be a definable group and $H$ be a normal subgroup.  A priori,
the quotient group $G/H$ is interpretable, not definable.  In this
section, we show that for certain \textit{dfg} groups $H$, the
quotient $G/H$ is automatically definable (Corollary~\ref{quots}).
The key is to show that certain definable types are coded by
\emph{real} tuples (Theorem~\ref{codes}).  Both of these results will
be proved in greater generality in future work \cite[Theorems~3.4, 4.1]{andujar-johnson}.

If $D$ is a definable set in a model $M$, let $\ulcorner D \urcorner$
denote ``the'' code of $D$ in $M^\eq$, which is well-defined up to
interdefinability.  If $\sigma \in \Aut(M)$, then
\begin{equation*}
  \sigma(D) = D \iff \sigma(\ulcorner D \urcorner) = \ulcorner D \urcorner,
\end{equation*}
and this property characterizes $\ulcorner D \urcorner$ when $M$ is
sufficiently saturated and homogeneous.
\begin{lemma} \label{var-codes}
  Let $K$ be a field and $V \subseteq K^n$ be Zariski closed.  Then
  the definable set $V$ is coded by a tuple in $K$ (rather than
  $K^\eq$).  In particular, finite subsets of $K^n$ are coded by
  tuples in $K$.
\end{lemma}
\begin{proof}
  Passing to an elementary extension, we may assume $K$ is
  $\aleph_1$-saturated and strongly $\aleph_1$-homogeneous.  Let $M =
  K^{alg}$.  Let $\overline{V}$ be the Zariski closure of $V$ in
  $M^n$.  Note $V = \overline{V} \cap K^n$.  By elimination of
  imaginaries in ACF, there is a tuple $b \in M$ which codes
  $\overline{V}$ in the structure $M^n$.  If $\sigma \in \Aut(M/K)$
  then $\sigma$ fixes $V$ setwise, so it also fixes the Zariski
  closure $\overline{V}$.  Therefore $\sigma(b) = b$, for any $\sigma
  \in \Aut(M/K)$.  By Galois theory, $b$ is in the perfect closure of
  $K$.  Replacing $b$ with $b^{p^n}$ if
  necessary, we may assume $b$ is a tuple in $K$.

  We claim that $b$ codes $V$ in the structure $K$.  Suppose $\sigma_0
  \in \Aut(K)$.  Extend $\sigma_0$ to an automorphism $\sigma \in
  \Aut(M)$ arbitrarily.  Then $b$ codes $V$ because
  \begin{equation*}
    \sigma_0(V) = V \iff \sigma(V) = V \overset{\ast}{\iff}
    \sigma(\overline{V}) = \overline{V} \iff \sigma(b) = b \iff
    \sigma_0(b) = b.
  \end{equation*}
  The starred $\overset{\ast}{\iff}$ requires some explanation.  The
  direction $\Rightarrow$ holds because the formation of Zariski
  closures is automorphism invariant.  The direction $\Leftarrow$
  holds because $\sigma$ fixes $K$ setwise and $V = \overline{V} \cap
  K^n$.
\end{proof}

\begin{lemma} \label{interdef}
Work in a monster model $\Mm$ of $\pCF$.
\begin{enumerate}
    \item If an imaginary tuple $a$ is algebraic over a real tuple $b$, then $a$ is definable over $b$.
    \item If an imaginary tuple $a$ is interalgebraic with a real tuple $b$, then $a$ is interdefinable with some real tuple $b'$.
\end{enumerate}
More generally, both statements hold if we work over a set of real parameters $C \subseteq \Mm$.
\end{lemma}
\begin{proof}
  \begin{enumerate}
      \item Note that $\dcl(b) \preceq \Mm$ by definable Skolem functions, and so $\dcl^\eq(b) \preceq \Mm^\eq$.  Submodels are algebraically closed, so $\acl^\eq(b) = \dcl^\eq(b)$ and $a \in \dcl^\eq(b)$.
      \item By part (1), $a \in \dcl^\eq(b)$.  Write $a$ as $f(b)$ for some $\varnothing$-definable function $f$.  Let $S \subseteq \Mm^n$ be the set of realizations of $\tp(b/a)$.  Then $S$ is finite as $b \in \acl^\eq(a)$.  Moreover, $S$ is $a$-definable, and so the code $\ulcorner S \urcorner$ is in $\dcl^\eq(a)$.  By Lemma~\ref{var-codes}, we can take the code $\ulcorner S \urcorner$ to be a real tuple.
      For any $c \in S$, we have $f(c) = a$, which implies $a \in \dcl^\eq(\ulcorner S \urcorner)$.  Then $a$ is interdefinable with the real tuple $\ulcorner S \urcorner$.
  \end{enumerate}
    The ``more general'' statements follow by the same proofs.  Indeed, we can name the elements of $C$ as constants without losing definable Skolem functions or codes for finite sets.
\end{proof}
If $p$ is a definable $n$-type over $M$, let
$\ulcorner p \urcorner$ denote the infinite tuple $(\ulcorner
D_\varphi \urcorner : \varphi \in L)$, where
\begin{equation*}
  D_\varphi = \{b \in M^m : \varphi(x,b) \in p(x)\}.
\end{equation*}
For $\sigma \in \Aut(M)$, we have
\begin{equation*}
  \sigma(p) = p \iff \sigma(\ulcorner p \urcorner) = \ulcorner p
  \urcorner,
\end{equation*}
and this property determines $\ulcorner p \urcorner$ up to
interdefinability when $M$ is sufficiently saturated and homogeneous.
\begin{lemma} \label{codes-1}
  If $q \in S_1(\Mm)$ is definable, then $\ulcorner q \urcorner$ is
  interdefinable with a (finite) real tuple.
\end{lemma}
\begin{proof}
  By \cite[Proposition~2.24]{johnson-yao}, the type $q$ must
  accumulate at some point $c$ in the projective line $\Pp^1(\Mm)$,
  because $\Pp^1(\Mm)$ is definably compact.  If necessary, we can
  push $q$ forward along the map $x \mapsto 1/x$ to ensure $c \ne
  \infty$.  Then $c \in \Mm$.  Note $c \in \dcl^\eq(\ulcorner q
  \urcorner)$.  There are only boundedly many types concentrating at
  $c$ by \cite[Corollary 7.5]{johnson-dpm1} or \cite[Fact~2.20]{johnson-yao}, so $\ulcorner q \urcorner$ has a small
  orbit under $\Aut(\Mm/c)$.  Then $\ulcorner q \urcorner \in
  \acl^\eq(c)$.  As in the proof of Lemma~\ref{interdef}(1), $\ulcorner q
  \urcorner \in \dcl^\eq(c)$, so $\ulcorner q \urcorner$ is interdefinable with $c$.
\end{proof}

\begin{theorem} \label{codes}
  Suppose $q \in S_n(\Mm)$ is a definable type, and $\dim(q) = 1$.
  Then $\ulcorner q \urcorner$ is interdefinable with a real tuple.
\end{theorem}
\begin{proof}
  Take an elementary extension $\Mm' \succeq \Mm$ containing a realization
  $\ba$ of $q$.  Then $\trdeg(\ba/\Mm) = \dim(q) = 1$, so there is
  some $i$ such that $a_i$ is a transcendence basis of $\ba$ over
  $\Mm$, implying that $\ba$ is field-theoretically algebraic over $\Mm$ and
  $a_i$.  Then there is a Zariski-closed set $V_0 \subseteq \Mm^n$ such
  that there are only finitely many $\bb \in V_0(\Mm')$ with $b_i =
  a_i$.

  Let $V \subseteq \Mm^n$ be the smallest Zariski-closed set such that
  $\ba \in V(\Mm')$, or equivalently, the smallest Zariski-closed set
  on which $q$ concentrates.  Any automorphism of $\Mm$ which fixes
  $q$ fixes $V$, and so
  \begin{equation}
    \ulcorner V \urcorner \in \dcl^\eq(\ulcorner q \urcorner). \label{eq1}
  \end{equation}
  As $V \subseteq V_0$, there are only finitely many $\bb \in V(\Mm')$
  with $b_i = a_i$.  Therefore $\ba \in \acl^\eq(\ulcorner V \urcorner
  a_i)$.  By Lemma~\ref{var-codes}, we may assume $\ulcorner V
  \urcorner$ is a real tuple in $\Mm$, and then $\ba \in
  \dcl^\eq(\ulcorner V \urcorner a_i)$ by
  Lemma~\ref{interdef}(1).  Therefore $\ba$ and $a_i$ are
  interdefinable over $\ulcorner V \urcorner$.

  Take a bijection $f$ defined over $\ulcorner V \urcorner$ such that
  $\ba = f(a_i)$.  Then $q = \tp(\ba/\Mm)$ is the pushforward of the
  definable type $r := \tp(a_i/\Mm)$ along the $\ulcorner V
  \urcorner$-definable bijection $f$.  Therefore
  \begin{gather}
    \ulcorner q \urcorner \in \dcl^\eq(\ulcorner V \urcorner \ulcorner r \urcorner) \label{eq2}
   \end{gather}
   Likewise, $r$ is the pushforward of $q$ along the 0-definable coordinate projection $\pi(\bx) = x_i$, so 
   \begin{gather}
    \ulcorner r \urcorner \in \dcl^\eq(\ulcorner q \urcorner) \label{eq3}
  \end{gather}
  Combining equations (\ref{eq1})--(\ref{eq3}), we see that $\ulcorner
  q \urcorner$ is interdefinable with $\ulcorner V \urcorner \ulcorner
  r \urcorner$.  But $\ulcorner V \urcorner$ is a real tuple by
  Lemma~\ref{var-codes} as noted above, and $\ulcorner r
  \urcorner$ is a real tuple by Lemma~\ref{codes-1}.
\end{proof}
Using a different argument, one can show that Theorem~\ref{codes}
holds for \emph{any} definable $n$-type, without the assumption
$\dim(q) = 1$ \cite[Theorem~3.4]{andujar-johnson}.  However, the real tuple may
need to be infinite \cite[Proposition~3.7]{andujar-johnson}.

\begin{proposition} \label{quotient}
If a one-dimensional \textit{dfg} group $G$ acts on a definable set $X$, then the quotient space $X/G$ is definable (not just interpretable).
\end{proposition}
\begin{proof}
  Take a global definable type $p$ on $G$ with boundedly many right translates.
  Take a small model $M_0$ over which everything is defined, including the boundedly many right translates of $p$.  It suffices to show that every element of the interpretable set $X/G$ is interdefinable over $M_0$ with a real tuple.  By Lemma~\ref{interdef}(2), it suffices to show that every element of $X/G$ is \emph{interalgebraic} over $M_0$ with a real tuple.  Fix some element $e = G \cdot a \in X/G$, where $a \in X$.  Let $p \cdot a$ denote the pushforward of $p$ along the map $x \mapsto x \cdot a$.  Note that the global types $p$ and $p \cdot a$ both have dimension 1 (or less).  By Theorem~\ref{codes}, the code $\ulcorner p \cdot a \urcorner$ can be taken to be a real tuple.  We claim that $\ulcorner p \cdot a \urcorner$ is interalgebraic with $e$ over $M_0$.
  
  In one direction, $p \cdot a$ is contained in the collection
  \begin{align*}
      \mathfrak{S} &= \{p \cdot a' : a' \in G \cdot a\} \\ &= \{p \cdot (g \cdot a) : g \in G\} = \{(p \cdot g) \cdot a : g \in G\},
  \end{align*}
  which is $\Aut(\Mm/M_0e)$-invariant by the first line, and small by the second line.  It follows that $p \cdot a$ has a small number of conjugates over $M_0 e$, and so $\ulcorner p \cdot a \urcorner \in \acl^\eq(M_0 e)$.
  
  In the other direction, the type $p \cdot a$ concentrates on $G \cdot a$, so its pushforward along the $M_0$-definable map $X \to X/G$ is the constant type $x = e$.  Therefore $e \in \dcl^\eq(M_0 \ulcorner p \cdot a \urcorner)$.  This completes the proof that $e$ is interalgebraic with $\ulcorner p \cdot a \urcorner$ over $M_0$.
\end{proof}
Again, this holds without the assumption $\dim(G) = 1$.  See
\cite[Theorem~4.1]{andujar-johnson}.
\begin{corollary} \label{quots}
  Let $G$ be a definable group and $H$ be a 1-dimensional
  definable normal subgroup.  If $H$ has \textit{dfg}, then $G/H$ is
  definable and $\dim(G/H) = \dim(G) - 1$.
\end{corollary}

\section{Proof of Theorem~\ref{th1}} \label{sec3}
Work in a model $M \models \pCF$.
\begin{theorem} \label{decomp}
  Let $M$ be a $p$-adically closed field and $G$ be a definable
  abelian group in $M$.  Then there is a definable short exact
  sequence
  \begin{equation*}
    1 \to H \to G \to C \to 1
  \end{equation*}
  such that $H$ has \textit{dfg}, $C$ has \textit{fsg}, and $C$ is
  definably compact.
\end{theorem}
\begin{proof}
  For definable groups, \textit{fsg} is equivalent to definable
  compactness \cite[Theorem~1.2]{johnson-fsg}.  Say a subgroup $H
  \subseteq G$ is ``good'' if $G/H$ is definable and $H$ has
  \textit{dfg}.  For example, $H = \{1\}$ is good.  Take a good subgroup $H$ maximizing $\dim(H)$.  If $G/H$ is
  definably compact then we are done.
  Otherwise, $G/H$ is not definably
  compact.  By \cite[Corollary~6.11]{johnson-yao}, there is a 1-dimensional  definable \textit{dfg}
  subgroup of $G/H$.  This subgroup has the form $H'/H$ for some
  definable subgroup of $H$.  The short exact sequence
  \begin{equation*}
    1 \to H \to H' \to H'/H \to 1
  \end{equation*}
  shows that $H'$ has \textit{dfg} by Lemma~\ref{dfg-ext}, and that
  \begin{equation*}
    \dim(H') = \dim(H) + \dim(H'/H) = \dim(H) + 1 > \dim(H).
  \end{equation*}
  The quotient $ G/H' = (G/H)/(H'/H)$ is definable by
  Corollary~\ref{quots}, and so $H'$ is a good subgroup,
  contradicting the choice of $H$.
\end{proof}

\section{Abelian groups over $\Qp$} \label{qp-case}

\begin{fact} \label{mos-fact}
Let $G$ be a definably amenable group definable over $\Qp$.  There is an algebraic group $H$ over $\Qp$ and a definable finite-to-one group homomorphism from $G^{00}$ to $H$.
\end{fact}
\begin{proof}
  This follows from \cite[Theorem~2.19]{MOS} via the proof of \cite[Corollary~2.22]{MOS}.
\end{proof}

\begin{theorem} \label{g00}
  If $G$ is an abelian group definable over $\Qp$, then $G^0 =
  G^{00}$.
\end{theorem}
\begin{proof}
  Theorem~\ref{decomp} gives a short exact sequence
  \begin{equation*}
    1 \to H \to G \to C \to 1
  \end{equation*}
  where $H$ has \textit{dfg} and $C$ is definably compact.  Then $C^0
  = C^{00}$ because $C$ is definably compact and defined over $\Qp$
  \cite[Corollary~2.4]{onshuus-pillay}, and $H^0 = H^{00}$ because $H$ is
  \textit{dfg} \cite[proof of Lemma~1.15]{pillay-yao0}.  Then $G^0 = G^{00}$ by
  Theorem~\ref{0-00-ext}.
\end{proof}

\begin{corollary}
  If $G$ is an abelian group definable in $\Qp$, then there is a
  finite index definable subgroup $E\sq G$ and finite subgroup $F$
  such that $E/F$ is isomorphic to an open subgroup of an algebraic
  group $A$ over $\Q$.
\end{corollary}
\begin{proof}
  By Theorem~\ref{g00}, $G^0 = G^{00}$.  By Fact~\ref{mos-fact}, there is an algebraic group $H$ and a finite-to-one definable homomorphism $f : G^0 \to H$.  By compactness there is a finite-index subgroup $E \subseteq G$ such that $f$ extends to a finite-to-one definable homomorphism $f' : E \to H$.  Replacing $H$ with the Zariski closure of the image of $f'$, we may assume the image is an open subgroup of $H$.
\end{proof}

\section{Interpretable groups} \label{sec5}
In this section, we discuss our original approach to
Theorem~\ref{decomp}, which yielded a weaker result, only
giving an interpretable group.  However, this approach is more general
in one way---one can \emph{start} with an interpretable group.  Unfortunately, in the interpretable case we don't know how to prove the termination of the recursive process implicit in the proof of Theorem~\ref{decomp}.

\begin{proposition} \label{quotop}
  Let $G$ be an abelian definable group, let $H$ be a definable
  subgroup, and let $X = G/H$ be the interpretable quotient group.
  Consider the canonical definable manifold topology on $G$, and the
  quotient topology on $X$.
  \begin{enumerate}
  \item The quotient map $\pi : G \to X$ is an open map.
  \item The quotient topology on $X$ is definable.
  \item The quotient topology on $X$ is a group topology.
  \item The quotient topology on $X$ is Hausdorff.
  \end{enumerate}
\end{proposition}
\begin{proof}
  \begin{enumerate}
  \item If $U \subseteq G$ is open, then $\pi^{-1}(\pi(U)) = U \cdot H
    = \bigcup_{h \in H} (U \cdot h)$ which is open.  By definition of
    the quotient topology, $\pi(U)$ is open.
  \item If $\mathcal{B}$ is a definable basis of opens on $G$, then
    $\{\pi(U) : U \in \mathcal{B}\}$ is a definable basis of opens on
    $X$, because $\pi$ is an open map.
  \item We claim $(x,y) \mapsto x \cdot y^{-1}$ is continuous on $X$.
    Fix $a, b \in X$.  Let $U \subseteq X$ be an open neighborhood of
    $a \cdot b^{-1}$.  Take $\tilde{a}, \tilde{b} \in G$ lifting $a$
    and $b$.  Then $\tilde{a} \cdot \tilde{b}^{-1} \in \pi^{-1}(U)$,
    which is open.  By continuity of the group operations on $G$,
    there are open neighborhoods $V \ni \tilde{a}$ and $W \ni
    \tilde{b}$ such that $x \in V, ~ y \in W \implies x \cdot y^{-1}
    \in \pi^{-1}(U)$.  Because $\pi$ is an open map, $\pi(V)$ and
    $\pi(W)$ are open neighborhoods of $a$ and $b$, respectively.  If
    $x \in \pi(V)$ and $y \in \pi(W)$, then $x \cdot y^{-1} \in U$,
    because we can write $x = \pi(\tilde{x}), ~ y = \pi(\tilde{y})$
    for $\tilde{x} \in V, ~ \tilde{y} \in W$, and then $x \cdot y^{-1}
    = \pi(\tilde{x} \cdot \tilde{y}^{-1}) \in \pi(\pi^{-1}(U)) = U$.
    This proves continuity of $x \cdot y^{-1}$ at $(a,b)$.
  \item Because the quotient topology is a group topology, it suffices
    to show that $\{1_X\}$ is closed.  By definition of the quotient topology, it suffices to show that $H$ is closed in $G$.  On definable manifolds, the frontier of a set
    is lower-dimensional than the set itself \cite[Theorem~3.5]{p-minimal-cells}:
    \begin{equation*}
      \dim(\overline{H} \setminus H) < \dim(H).
    \end{equation*}
    But $\overline{H} \setminus H$ is a union of cosets of $H$, and each coset has dimension $\dim(H)$.  Therefore
    $\overline{H} \setminus H$ must be empty, and $H$ is
    closed. \qedhere
  \end{enumerate}
\end{proof}

\begin{definition} \label{mdgroup}
  A \emph{manifold-dominated group} is an interpretable group $X$ with
  a Hausdorff definable group topology such that there is a definable manifold
  $\tilde{X}$ and an interpretable surjective continuous open map $f :
  \tilde{X} \to X$.
\end{definition}
In the setting of Proposition~\ref{quotop}, $X$ is manifold
dominated via the map $G \to X$.
\begin{remark}
  If $X$ is \emph{any} interpretable group, then there is a definable
  group topology $\tau$ on $X$ making $(X,\tau)$ be manifold-dominated
  \cite[Theorem~5.10]{admissible}.  Moreover, $\tau$ is uniquely
  determined, though the manifold $\tilde{X}$ is not.  This motivates
  working in the more general context of manifold-dominated abelian
  groups, rather than the special case of quotient groups $G/H$.
\end{remark}
\begin{theorem} \label{int-ps}
  Let $X$ be a manifold-dominated interpretable abelian group.
  Suppose $X$ is not definably compact.  Then there is an
  interpretable subgroup $X' \subseteq X$ with the following
  properties:
  \begin{enumerate}
  \item $X'$ is not definably compact.
  \item $\dpr(X') = 1$.
  \item $X'$ has \textit{dfg}.
  \end{enumerate}
\end{theorem}
Theorem~\ref{int-ps} is an analogue of \cite[Theorem~6.8,
  Corollary~6.11]{johnson-yao}, and the proof is similar.
Nevertheless, we sketch the proof for completeness.

For the rest of the section, work in a monster model $\Mm$.  Fix a definable manifold $\tilde{X}$, an interpretable abelian group $X$ with a Hausdorff definable group
topology, and an interpretable continuous surjective open map $\pi :
\tilde{X} \to X$.  Also fix a small model $K$ over which everything is defined.
\begin{definition}
  If $S$ is an interpretable topological space (in $\pCF$) and $x_0
  \in S$, then a \emph{good neighborhood basis} of $x_0$ is an
  interpretable family $\{O_t\}_{t \in \Gamma}$ with the following
  properties:
  \begin{enumerate}
  \item $\{O_t\}_{t \in \Gamma}$ is a neighborhood basis of $x_0$.
  \item $t \le t' \implies O_t \subseteq O_{t'}$.
  \item Each set $O_t$ is clopen and definably compact.
  \item $\bigcup_t O_t = S$.
  \end{enumerate}
\end{definition}
This is more general than the definition in
\cite[Definition~2.27]{johnson-yao}, since we are considering
topological spaces rather than topological groups.  The
definition here is slightly weaker, since we do not require $O_t^{-1}
= O_t$ when $S$ is a group.

Fix some element $\tilde{1} \in \tilde{X}$ lifting $1 \in X$.  By the
proof of \cite[Proposition~2.28]{johnson-yao}, there is a good
neighborhood basis $\{O_t\}_{t \in \Gamma}$ of $\tilde{1}$ in
$\tilde{X}$.  Let $V_t = \pi(O_t)$.  Then $\{V_t\}_{t \in \Gamma}$ is a good neighborhood
basis of $1$ in $X$.  The analogue of
\cite[Proposition~2.29]{johnson-yao} holds, via the same proof:
\begin{enumerate}
\item For any $t \in \Gamma$, there is $t' \in \Gamma$ such that
  $V_{t'} \cdot V_{t'}^{-1} \subseteq V_t$.
\item For any $t \in \Gamma$, there is $t'' \in \Gamma$ such that $V_t
  \cdot V_t^{-1} \subseteq V_{t''}$.
\end{enumerate}
Say that a set $S \subseteq X$, not necessarily interpretable, is
\emph{bounded} if $S \subseteq V_t$ for some $t \in \Gamma$.  As in
\cite[Proposition~2.10]{johnson-yao}, $S$ is bounded if and only if
$S$ is contained in a definably compact subset of $X$.
If $A, B \subseteq X$, let $A \diamond B$ denote the set
\begin{equation*}
  \{g \in A : gB \cap A = \varnothing\},
\end{equation*}
as in \cite[\S4.1]{johnson-yao}.  Let $A \diamond B \setminus C$ mean
$A \diamond (B \setminus C)$.
\begin{lemma}\label{old-7}
Let $I \subseteq X$ be an unbounded interpretable set.  Let $A \subseteq X$ be bounded, but not necessarily interpretable.  Then there is $t \in \Gamma_{\Mm}$ such that $I \diamond V_t
  \setminus A$ is bounded.
\end{lemma}
\begin{proof}
  The proofs of Lemmas~4.9, 4.10, 4.11 in \cite{johnson-yao} work
  here, after making a couple trivial changes.  The interpretable
  group $X$ has finite dp-rank because $\dpr(X) \le \dpr(\tilde{X}) =
  \dim(\tilde{X}) < \infty$.
\end{proof}
Recall our assumption that $\pi : \tilde{X} \to X$ is $K$-interpretable for some small model $K$.  Fix $|K|^+$-saturated $L$ with $K \preceq L \preceq \Mm$.  If $\Sigma$ is a
definable type or definable partial type over $K$, then $\Sigma^L$ denotes its
canonical extension over $L$.  (See \cite[Definition~2.12]{PS} for
definability of partial types.  When $\Sigma$ is complete, $\Sigma^L$
is the heir of $\Sigma$.)
\begin{lemma} \label{pq}
  There is a 1-dimensional definable type $p \in S_{\tilde{X}}(K)$
  whose pushforward $q = \pi_* p$ has the following properties:
  \begin{enumerate}
  \item $q$ is ``unbounded'' over $K$, in the sense that $q$ does not
    concentrate on any $K$-interpretable bounded set, or equivalently, $q$ does
    not concentrate on $V_t$ for any $t \in \Gamma_K$.
  \item Similarly, the heir $q^L$ is unbounded over $L$.
  \item If $b \in X$ realizes $q$ and $b \notin V_t$ for any
    $t \in \Gamma_L$, then $b$ realizes $q^L$.
  \end{enumerate}
\end{lemma}
\begin{proof}
  Take $u \in \Mm$ with $v(u) > \Gamma_K$.  In other words, $u$ is
  infinitesimally close to 0 over $K$.  Then $\tp(u/K)$ is definable.
  Let $\gamma = v(u)$.  As $X$ is not definably compact, $V_\gamma \ne
  X$.  The set $\pi^{-1}(X \setminus V_\gamma)$ is a non-empty
  $Ku$-definable subset of $\tilde{X}$.  By definable Skolem
  functions, there is $\beta_0 \in \pi^{-1}(X \setminus V_\gamma)$
  with $\beta_0 \in \dcl(Ku)$.  Then $\beta_0 = f(u)$ for some
  $K$-definable function $f$.  Let $p = \tp(\beta_0/K)$.  Then $p =
  f_*(\tp(u/K))$, so $p$ is definable.  Let $b_0 = \pi(\beta_0)$ and
  let $q = \pi_* p = \tp(b_0/K)$.  By choice of $\beta_0$, $b_0 =
  \pi(\beta_0) \notin V_\gamma$, which implies $b_0 \notin V_t
  \subseteq V_\gamma$ for any $t \in \Gamma_K$.  Thus $q$ is unbounded
  over $K$.  As $q^L$ is the heir, it is similarly unbounded over $L$.

  Finally, suppose that $b$ satisfies the assumptions of (3).  Then
  $\tp(b/K) = q = \tp(b_0/K)$, so there is $\sigma \in \Aut(\Mm/K)$
  with $\sigma(b_0) = b$.  Let $\beta = \sigma(\beta_0)$.  Then
  $(b,\beta) \equiv_K (b_0,\beta_0)$, and in particular $\beta$
  realizes $p$ and $\pi(\beta) = b$.  Recall the sets $O_t$ used to define $V_t$.  If $\beta \in O_t$ for some $t \in
  \Gamma_L$, then $b = \pi(\beta) \in \pi(O_t) = V_t$, contradicting
  the assumptions.  Therefore, $\beta \notin O_t$ for any $t \in
  \Gamma_L$.  By \cite[Lemma~2.25]{johnson-yao}, $\beta$ realizes $p^L$.
  Then $b = \pi(\beta)$ realizes $\pi_* (p^L) = q^L$.
\end{proof}
Fix $p, q$ as in Lemma~\ref{pq}.  Fix $\beta \in \tilde{X}$ realizing
$p^L$ and let $b = \pi(\beta) \in X$.  Then $b$ realizes $q^L$.

We will make use of the notation and facts from
\cite[\S5]{johnson-yao}, applied to the group $X$ and the definable
type $q$.  In particular, $\mu$ is the infinitesimal partial type of $X$ over
$K$, $\mu^L$ is the infinitesimal partial type of $X$ over $L$, and
$\st^{\Mm}_L$ is the standard part map, a partial map from $X$ to
$X(L)$.  The domain of $\st^{\Mm}_L$ is the subgroup $\mu^L(\Mm) \cdot
X(L)$ of points in $X$ infinitesimally close to points in $X(L)$.  If
$Y \subseteq X$, then $\st^{\Mm}_L(Y)$ denotes the image of $Y \cap
(\mu^L(\Mm) \cdot X(L))$ under $\st^{\Mm}_L$.

The following lemma takes the place of \cite[Fact~6.3]{johnson-yao}.
\begin{lemma} \label{hard}
  Suppose $Y \subseteq X$ is $\beta$-interpretable.
  \begin{enumerate}
  \item The set $\st^{\Mm}_L(Y) \subseteq X(L)$ is interpretable (in the
    structure $L$)
  \item $\dpr(\st^{\Mm}_L(Y)) \le \dpr(Y)$.
  \end{enumerate}
\end{lemma}
See Remark~\ref{ict-remark} for the definition of \emph{ict pattern} and \emph{dp-rank}.
\begin{proof}
  \begin{enumerate}
  \item Fix some interpretable basis of opens for $X$.  Let $\mathcal{F}$
    be the collection of $L$-interpretable basic open sets which intersect
    $Y$.  Then $\mathcal{F}$ is interpretable in the structure $L$,
    because $\mathcal{F}$ is defined externally using $\beta$, but
    $\tp(\beta/L)$ is definable.  Now if $a \in X(L)$, the following
    are equivalent:
    \begin{enumerate}
    \item $a \in \st^{\Mm}_L(Y)$.
    \item There is $a' \in Y$ such that for every $L$-interpretable basic
      open neighborhood $U \ni a$, we have $a' \in U$.
    \item For every $L$-interpretable basic open neighborhood $U \ni a$,
      there is $a' \in Y$ such that $a' \in U$.
    \item Every $L$-interpretable basic open neighborhood of $a$ is in
      $\mathcal{F}$.
    \end{enumerate}
    Indeed, (a)$\iff$(b) by definition, (b)$\iff$(c) by saturation of
    $\Mm$, and (c)$\iff$(d) by definition of $\mathcal{F}$.  Condition
    (d) is definable because $\mathcal{F}$ is.
  \item Let $r$ be the dp-rank of the interpretable set $D :=
    \st^{\Mm}_L(Y)$.  It is finite, bounded by $\dpr(X)$.  There is an
    ict-pattern of depth $r$ in $D$.  That is, there are are uniformly
    interpretable sets $S_{i,j} \subseteq D$ for $i < r$ and $j < \omega$,
    and points $b_\eta \in D$ for $\eta \in \omega^r$, such that
    $b_\eta \in S_{i,j} \iff j = \eta(i)$.  By Theorem~\ref{nice-ict}
    in the appendix, we can also ensure that $S_{i,j}$ is open and $j
    \ne \eta(i) \implies b_\eta \notin \overline{S_{i,j}}$.  As $L$ is
    $\aleph_1$-saturated, we can arrange for all the data to be
    $L$-interpretable.  Then each $b_\eta$ is $\st^{\Mm}_L(b'_\eta)$ for
    some $b'_\eta \in Y$.  Since $S_{i,j}$ is open and $L$-interpretable,
    we have $b'_\eta \in S_{i,j}$ for $j = \eta(i)$.  Since
    $\overline{S_{i,j}}$ is closed and $L$-interpretable, we have $b'_\eta
    \notin \overline{S_{i,j}}$ for $j \ne \eta(i)$.  Then the sets
    $S_{i,j}$ and elements $b'_\eta$ are an ict-pattern of depth $r$
    in $Y$, showing $\dpr(Y) \ge r = \dpr(D)$. \qedhere
  \end{enumerate}
\end{proof}
\begin{lemma} \label{equals}
  The following subsets of $X(L)$ are equal:
  \begin{enumerate}
  \item $\stab(\mu^L \cdot q^L)$.
  \item $\bigcap_{\varphi \in \mathcal{L}} \stab_\varphi(\mu \cdot q)(L)$.
  \item $\st^{\Mm}_L(q^L(\Mm) b^{-1})$
  \item $\bigcap_{\psi \in q^L} \st^{\Mm}_L(\psi(\Mm) b^{-1})$
  \item $\bigcap_{\psi \in q} \st^{\Mm}_L(\psi(\Mm) b^{-1})$.
  \end{enumerate}
\end{lemma}
See \cite[Definition~5.3]{johnson-yao} for the definition of $\stab_\varphi(-)$.
\begin{proof}
  The equivalence of (1)--(4) is Remark~5.12 and Lemma~5.13 in
  \cite{johnson-yao}.  The equivalence of (4) and (5) follows by a
  similar argument to the proof of \cite[Lemma~6.2]{johnson-yao},
  using Lemma~\ref{pq}(3) instead of \cite[Lemma~2.25]{johnson-yao}.
\end{proof}

\begin{lemma} \label{key}
  If $I \subseteq X$ is $L$-interpretable and contains $b$, then
  $\st^{\Mm}_L(Ib^{-1})$ is unbounded in $X(L)$.
\end{lemma}
\begin{proof}
  If not, take $t \in \Gamma_L$ such that $\st^{\Mm}_L(Ib^{-1})
  \subseteq V_t$.  By Lemma~\ref{pq}(2), $b$ is not in any
  $L$-interpretable bounded sets.  Therefore $I$ is unbounded.  By
  Lemma~\ref{old-7}, we can find $t' \in \Gamma_L$ such that $I
  \diamond V_{t'} \setminus V_t$ is bounded.  Then $b \notin
  I \diamond V_{t'} \setminus V_t$.  This means that
  \begin{equation*}
    b \cdot (V_{t'} \setminus V_t) \cap I \ne \varnothing.
  \end{equation*}
  Therefore there is $a \in V_{t'} \setminus V_t$ such that $ba \in
  I$.  Then there is $\alpha \in O_{t'}$ with $\pi(\alpha) = a$.  The
  conditions on $\alpha$ and $a$ are definable over $\dcl(Lb)
  \subseteq \dcl(L\beta)$ (where $\beta$ is the realization of $p^L$).
  By definable Skolem functions, we can assume $\alpha \in
  \dcl(L\beta)$.  Then $\tp(\alpha/L)$ is a pushforward of
  $\tp(\beta/L)$, so $\tp(\alpha/L)$ is a 1-dimensional definable type
  on $\tilde{X}$.  This type $\tp(\alpha/L)$ concentrates on the
  definably compact set $O_{t'} \subseteq \tilde{X}$, and therefore
  $\tp(\alpha/L)$ specializes to some point $\gamma \in G(L)$ by
  \cite[Lemma~2.23]{johnson-yao}.  As the map $\pi : \tilde{X} \to X$
  is continuous, $\tp(a/L)$ specializes to $c := \pi(\gamma) \in
  X(L)$.  Thus $\st^{\Mm}_L(a)$ exists and equals $c$.  Since $V_{t'}
  \setminus V_t$ is closed, $\st^{\Mm}_L(a) \in V_{t'} \setminus V_t$.
  But $a \in b^{-1}I = Ib^{-1}$, and
  \begin{equation*}
    \st^{\Mm}_L(a) \in \st^{\Mm}_L(Ib^{-1}) \subseteq V_t,
  \end{equation*}
  a contradiction.
\end{proof}

We can now complete the proof of Theorem~\ref{int-ps}.  By
Lemma~\ref{equals},
\begin{equation*}
  \bigcap_{\varphi \in \mathcal{L}} \stab_\varphi(\mu \cdot q)(L) =
  \bigcap_{\psi \in q} \st^{\Mm}_L(\psi(\Mm) b^{-1}). \tag{$\ast$}
\end{equation*}
The groups $\stab_\varphi(\mu \cdot q)$ are $K$-interpretable because $\mu
\cdot q$ is a $K$-definable partial type.  The sets $\st^{\Mm}_L(\psi(\Mm)
b^{-1})$ are interpretable by Lemma~\ref{hard}(1).  Both
intersections involve at most $|K|$ terms, and both intersections are
filtered.

If some $\stab_\varphi(\mu \cdot q)(L)$ is bounded, then by
$|K|^+$-saturation of $L$ we have $\st^{\Mm}_L(\psi(\Mm) b^{-1})
\subseteq \stab_\varphi(\mu \cdot q)(L)$ for some $\psi(x) \in q(x)$,
contradicting Lemma~\ref{key}.  Therefore, every group
$\stab_\varphi(\mu \cdot q)(L)$ is unbounded.  Consequently, \emph{no}
$\stab_\varphi(\mu \cdot q)$ is definably compact.

Since $\tp(\beta/K)$ has dimension 1, there is some $K$-definable set
$D \ni \beta$ of dimension 1.  Then $\dpr(\pi(D)) \le \dpr(D) =
\dim(D) = 1$.  If $\psi(x)$ defines $\pi(D)$, then $\psi(x) \in q =
\tp(b/K)$, and $\st^{\Mm}_L(\psi(\Mm) b^{-1})$ has dp-rank at most
1 by Lemma~\ref{hard}(2).  By $|K|^+$-saturation,
($\ast$) gives some $\varphi$ such that $\stab_\varphi(\mu \cdot q)(L)
\subseteq \st^{\Mm}_L(\psi(\Mm) b^{-1})$.  Then $\stab_\varphi(\mu
\cdot q)$ has dp-rank at most 1.  On the other hand,
$\stab_\varphi(\mu \cdot q)$ is infinite, since it is not definably
compact.  Therefore $X' := \stab_\varphi(\mu \cdot q)$ has dp-rank at least
1.

It remains to show that the interpretable subgroup $X' \subseteq X$ has
\textit{dfg}.  The proof of \cite[Lemma~6.10]{johnson-yao} works with
minor changes.  For completeness, we give the details.  For abelian
groups of dp-rank 1, ``not \textit{fsg}'' implies \textit{dfg} as in
the proof of \cite[Lemma~2.9]{pillay-yao}.  It suffices to show that $X'$ does \emph{not}
have \textit{fsg}.  Assume for the sake of contradiction that $X'$ has
\textit{fsg}.  By \cite[Proposition~4.2]{HPP}, non-generic sets form
an ideal, and there is a small model $M_0$ such that every generic set
contains an $M_0$-point.  Take $t$ large enough that $V_t$ contains
every point in $X(M_0)$.  Then $X' \setminus V_t$ is not generic in
$X'$, so $X' \cap V_t$ \emph{is} generic, meaning that finitely many
translates of $X' \cap V_t$ cover $X'$.  But $X' \cap V_t$ and its
translates are bounded (as subsets of $X$), so then $X'$ is bounded, a
contradiction.  This completes the proof of
Theorem~\ref{int-ps}.
\begin{corollary}\label{int-main}
  Let $X$ be an abelian interpretable group.  Then there is $\alpha \le
  \omega$ and an increasing chain of \textit{dfg} subgroups $(Y_i : i < \alpha)$ with $Y_0 = 0$
  such that the quotients $Y_i/Y_{i+1}$ have dp-rank 1.  In the case
  when $\alpha < \omega$, the quotient $X/Y_{\alpha-1}$ is definably compact and
  has \textit{fsg}.
\end{corollary}
\begin{proof}
  Any interpretable group is manifold-dominated
  \cite[Theorem~5.10]{admissible}, so we can apply
  Theorem~\ref{int-ps} to any interpretable group.  The
  first application gives $Y_1$; applying the theorem to $X/Y_1$ gives
  $Y_2$, and so on.  The process terminates if any quotient $X/Y_i$ is
  definably compact.  Definably compact groups have \textit{fsg}
  \cite[Theorem~7.1]{admissible}.  To prove that the groups $Y_i$ have
  \textit{dfg}, we can no longer use Lemma~\ref{dfg-ext}, as
  $\pCF^\eq$ lacks definable Skolem functions.  But
  Theorem~\ref{dfg-ext-2} in the appendix works.
\end{proof}
\begin{remark}
  If we start with a quotient group $G/H$, we can replace the use of
  \cite[Theorem~5.10]{admissible} with
  Proposition~\ref{quotop} above.
\end{remark}
\begin{remark}
  If $X$ is \emph{definable}, then the quotients $Y_i/Y_j$ are
  definable by induction on $i-j$, using Corollary~\ref{quots}.
  Then $\dim(Y_{i+1}/Y_i) = \dpr(Y_{i+1}/Y_i) = 1$, which implies
  $\dim(Y_{i+1}) > \dim(Y_i)$.  Therefore, the sequence \emph{must}
  terminate, as we saw in the proof of Theorem~\ref{decomp}.
  In the general interpretable case, it's unclear whether this works,
  so we make a conjecture:
\end{remark}
\begin{conjecture} \label{conj}
  In Corollary~\ref{int-main}, $\alpha$ is finite.
  Therefore, any abelian interpretable group $X$ sits in a short exact
  sequence $1 \to Y_{\alpha-1} \to X \to X/Y_{\alpha-1} \to 1$ where $Y_{\alpha-1}$
  has \textit{dfg} and $X/Y_{\alpha-1}$ has \textit{fsg} and is definably
  compact.
\end{conjecture}
Pillay and Yao asked whether any definably amenable group $G$ in a distal theory sits in a short exact sequence $1 \to H \to G \to C \to 1$ with $C$ having \textit{fsg} and $H$ having \textit{dfg} \cite[Question~1.19]{pillay-yao0}.  If Conjecture~\ref{conj} is true, it would provide further evidence for this.

\appendix

\section{Nice ict patterns} \label{ict}
\begin{remark} \label{ict-remark}
Following \cite[Definition~4.21]{NIPguide}, an \emph{ict-pattern} of depth $\kappa$ in a partial type $\Sigma(x)$ is a sequence of formulas $\varphi_i(x;y_i)$ and an array $(b_{i,j} : i <\kappa, ~ j < \omega)$ with $|b_{i,j}| = |y_i|$ such that for any function $\eta : \kappa \to \omega$, the following partial type is consistent:
\[ \Sigma(x) \cup \{\varphi_{i,\eta(i)}(x,b_{i,\eta(i)}) : i < \kappa\} \cup \{\neg \varphi_{i,j}(x,b_{i,j}):i < \kappa, ~ j \ne \eta(i)\}\]
Abusing notation, we say that $(\varphi_i(x;b_{i,j}) : i < \kappa, ~ j < \omega)$ is an ict-pattern to mean that the pair $((\varphi_i : i < \kappa),(b_{i,j} : i <\kappa, ~ j <\omega))$ is an ict-pattern.  Sometimes we consider ict-patterns where the columns are indexed by an infinite linear order $I$ other than $\omega$.  The definition is analogous, and ict-patterns of this sort can be converted to ict-patterns indexed by $\omega$ via a compactness argument.

Finally, the \emph{dp-rank} of $\Sigma(x)$ is the supremum of cardinals $\kappa$ such that there is an ict-pattern of depth $\kappa$ in $\Sigma(x)$, possibly in an elementary extension.
\end{remark}

Work in $\Mm^\eq$ for some monster model $\Mm \models \pCF$.  There is
a well-behaved notion of dimension on $\Mm^\eq$ \cite{gagelman}, which
gives rise to a notion of independence:
\begin{equation*}
  a \dimind_C b \iff \dim(a/Cb) = \dim(a/C) \iff \dim(b/Ca) =
  \dim(b/C).
\end{equation*}
This notion satisfies many of the usual properties
\cite[\S2.1]{admissible}.\footnote{The one unusual property is that ``$\dim(a/C) = 0$'' is strictly weaker than ``$a \in \acl(C)$''.} Say that a sequence $\{a_i : i \in I\}$ is
\emph{dimensionally independent} over a set $B$ if $a_i \dimind_B
a_{<i}$ for $i \in I$, where $a_{<i} = \{a_j : j < i\}$.  As usual,
this is independent of the order on $I$.
\begin{lemma}
  If $\tp(a/Cb)$ is finitely satisfiable in $C$, then $a \dimind_C
  b$.
\end{lemma}
\begin{proof}
  Suppose not.  Let $n = \dim(b/Ca) < \dim(b/C)$.  By
  \cite[Proposition~3.7]{gagelman}, there is a $Ca$-interpretable set $X$
  containing $b$ with $\dim(X) = n$.  Write $X$ as $\varphi(a,\Mm)$
  for some $\mathcal{L}^\eq_C$-formula $\varphi(x,y)$.  By
  \cite[Proposition~2.12]{admissible}, the set $\{a' \in \Mm :
  \dim(\varphi(a',\Mm)) = n\}$ is definable, defined by some
  $\mathcal{L}^\eq_C$-formula $\psi(x)$.  Then $\Mm \models \varphi(a,b) \wedge
  \psi(a)$.  As $\tp(a/Cb)$ is finitely satisfiable in $C$, there is
  some $a' \in C$ such that $\Mm \models \varphi(a',b) \wedge
  \psi(a')$.  Then $b$ is in the $C$-interpretable set $\varphi(a',\Mm)$
  which has dimension $n$ as $\Mm \models \psi(a')$.  Therefore
  $\dim(b/C) \le n$, a contradiction.
\end{proof}
\begin{corollary} \label{naming}
  Suppose $\ldots, b_{-1}, b_0, b_1, \ldots, \ldots, c_{-1}, c_0, c_1, \ldots$ is
  $C_0$-indiscernible.  Then the sequence $\ldots, b_{-1}, b_0, b_1, \ldots$ is dimensionally
  independent over $C = C_0 \cup \{c_i : i \in \Zz\}$.
\end{corollary}
\begin{proof}
  For example, $p = \tp(b_n/Cb_1 b_2 \cdots b_{n-1})$ is finitely satisfiable in $C$; any formula in $p$ is satisfied by $c_i$ for $i \ll 0$.  This argument shows that any finite subsequence of $\{b_i\}_{i \in \Zz}$ is dimensionally independent over $C$.  This implies the full sequence is dimensionally independent, by finite character of $\dimind$.
\end{proof}
\begin{lemma} \label{discard}
  If $\{b_i : i \in I\}$ is dimensionally independent over $C$, and
  $\dim(a/C) = n$, then $a \dimind_C b_i$ for all but at most $n$
  values of $i$.
\end{lemma}
The proof is standard, but we include it for completeness.
\begin{proof}
  Otherwise, passing to a subsequence, we could arrange for
  $b_1,\ldots,b_{n+1}$ to be dimensionally independent over $C$, but
  $a {\centernot\ind}^{\dim}_C b_i$ for each $i$.  The sequence
  $(\dim(a/Cb_1,\ldots,b_i) : 0 \le i \le n+1)$ cannot decrease $n+1$
  times, so there is some $0 \le i \le n$ such that
  $\dim(a/Cb_1,\ldots,b_i) = \dim(a/Cb_1,\ldots,b_{i+1})$, i.e.,
  \begin{equation*}
    a \ind^{\dim}_{Cb_1,\ldots,b_i} b_{i+1}.
  \end{equation*}
  As $b_1,\ldots,b_i \ind^{\dim}_C b_{i+1}$, left transitivity gives
  $a \ind^{\dim}_C b_{i+1}$, a contradiction.
\end{proof}
\begin{lemma} \label{indep-ict}
  Let $X$ be a $C$-interpretable set of parameters, with dp-rank $r$.
  Then there is $C' \supseteq C$ and an ict pattern of depth $r$ in
  $X$ of the form $(\varphi_i(x;b_{i,j}) : i < r, ~ j \in \Zz)$, such
  that the array $(b_{i,j} : i < r, ~ j \in \Zz)$ is mutually
  $C'$-indiscernible, and for each $i$, the sequence $(b_{i,j} : j \in
  \Zz)$ is dimensionally independent over $C'$.
\end{lemma}
\begin{proof}
  Let $\Zz + \Zz'$ denote two copies of $\Zz$ laid end to end, with
  the second copy denoted $\Zz'$.  Take an ict pattern
  $(\varphi_i(x;b^0_{i,j}) : i < r, ~ j < \omega)$ in $X$.  Let
  $(b_{i,j} : i < r, ~ j \in \Zz + \Zz')$ be a mutually
  $C$-indiscernible array extracted from $(b^0_{i,j} : i < r, ~ j <
  \omega)$.  Then $(\varphi_i(x;b_{i,j}) : i < r, ~ j \in \Zz + \Zz')$
  is an ict pattern in $X$.  Let $C' = C \cup \{b_{i,j}, i < r, ~ j
  \in \Zz'\}$.  Then $(b_{i,j} : i < r, ~ j \in \Zz)$ is mutually
  $C'$-indiscernible, and each row is dimensionally independent over
  $C'$ by Corollary~\ref{naming}.
\end{proof}

\begin{theorem} \label{nice-ict}
  Let $G$ be a manifold-dominated interpretable group of dp-rank $r$.
  There is an ict-pattern $(\varphi_i(x;b_{i,j}) : i < r, ~ j <
  \omega)$ in $G$ such that if $S_{i,j} = \varphi_i(\Mm;b_{i,j})$,
  then the following properties hold:
  \begin{enumerate}
  \item Each set $S_{i,j}$ is open.
  \item For each function $\eta : r \to \omega$, there is an element
    $a_\eta \in G$ such that
    \begin{gather*}
      j = \eta(i) \implies a_\eta \in S_{i,j} \\
      j \ne \eta(i) \implies a_\eta \notin \overline{S_{i,j}}
    \end{gather*}
  \end{enumerate}
\end{theorem}
\begin{proof}
  By \cite[Theorem~5.10]{admissible}, the topology on $G$ is ``admissible'', and so
  \begin{equation*}
    \dim(\overline{D} \setminus D) < \dim(D) \tag{Small boundaries property}
  \end{equation*}
  for any interpretable subset $D \subseteq G$, by
  \cite[Proposition~4.34]{admissible}.  By Lemma~\ref{indep-ict},
  there is an ict-pattern $(\psi_i(x;b_{i,j}) : i < r, ~ j \in \Zz)$
  and a set of parameters $C$ (over which $G$ is interpretable) such
  that the $b_{i,j}$ are mutually indiscernible over $C$, and each row
  is dimensionally independent over $C$.  Take some $a$ such that $\Mm \models
  \psi_i(a;b_{i,j}) \Leftrightarrow j = 0$ for all $i < r$ and $j \in \Zz$.  By
  \cite[Proposition~3.7]{gagelman} there is a formula
  $\theta_i(x;b_{i,0},c_i)$ in $\tp(a/Cb_{i,0})$ such that
  $\dim(\theta_i(x;b_{i,0},c_i)) = \dim(a/Cb_{i,0})$.  Replacing
  $b_{i,j}$ with $b_{i,j}c_i$ and replacing $\psi_i(x;b_{i,j})$ with
  $\psi_i(x;b_{i,j}) \wedge \theta_i(x;b_{i,j},c_i)$, we may assume
  that $\dim(\psi_i(x;b_{i,0})) = \dim(a/Cb_{i,0}) =: k_i$.  Let
  $V_{i,j} = \psi_i(\Mm;b_{i,j})$.  Then $\dim(V_{i,j}) =
  \dim(V_{i,0}) = k_i$ by indiscernibility.

  For each $i$, we have $a \dimind_C b_{i,j}$ for all but finitely
  many $j$, by Lemma~\ref{discard}.  Throwing away the finitely many
  bad values of $b_{i,j}$ in each row, we may assume $a \dimind_C b_{i,j}$ for
  all $j \ne 0$.  Thus $\dim(a/Cb_{i,j}) = \dim(a/C)$ for $j \ne 0$.
  By the Small Boundaries Property,
  \begin{equation*}
    \dim(\overline{V_{i,j}} \setminus V_{i,j}) < \dim(V_{i,j}) = \dim(V_{i,0}) = k_i = \dim(a/Cb_{i,0}) \le \dim(a/C) = \dim(a/Cb_{i,j}),
  \end{equation*}
  for $j \ne 0$.  Then $a$ cannot be in the $Cb_{i,j}$-interpretable set
  $\overline{V_{i,j}} \setminus V_{i,j}$.  By choice of $a$, we also
  have $a \notin V_{i,j}$.  So $a \notin \overline{V_{i,j}}$ for any
  $j \ne 0$.    Thus
  \begin{gather*}
    j = 0 \implies a \in V_{i,j} \\
    j \ne 0 \implies a \notin \overline{V_{i,j}}.
  \end{gather*}
  By mutual indiscernibility, we can find $a_\eta$ for any $\eta : r
  \to \Zz$ such that
  \begin{gather*}
    j = \eta(i) \implies a_\eta \in V_{i,j} \\
    j \ne \eta(i) \implies a_\eta \notin \overline{V_{i,j}}.
  \end{gather*}
  Recall that the topology on $G$ is a group topology, so every open
  neighborhood of $a_\eta$ has the form $a_\eta \cdot N$ for some open
  neighborhood $N$ of $1$.  For each $i,j,\eta$ with $j \ne \eta(i)$,
  we can find an open neighborhood $N_{i,j,\eta} \ni 1$ such that
  $(a_\eta \cdot N_{i,j,\eta}) \cap V_{i,j} = \varnothing$.  By
  saturation, there is an interpretable open neighborhood $N_0 \ni 1$ with $N_0
  \subseteq N_{i,j,\eta}$ for all $i,j,\eta$.  Because the topology is
  a group topology, there is a smaller interpretable open neighborhood $N \ni 1$ such
  that $N = N^{-1}$ and $N \cdot N \subseteq N_0$.

  Let $U_{i,j} = V_{i,j} \cdot N = \{x \cdot y : x \in V_{i,j}, ~ y
  \in N\}$.  Note that $U_{i,j}$ is open.  If $j \ne \eta(i)$, then
  \begin{equation*}
    (a_\eta \cdot N \cdot N) \cap V_{i,j} \subseteq a_\eta \cdot
    N_{i,j,\eta} \cap V_{i,j} = \varnothing.
  \end{equation*}
  The fact that $(a_\eta \cdot N \cdot N) \cap V_{i,j} = \varnothing$
  implies that \[(a_\eta \cdot N) \cap U_{i,j} = (a_\eta \cdot N) \cap
  (V_{i,j} \cdot N) = \varnothing.\] The neighborhood $a_\eta \cdot N$
  then shows that $a_\eta \notin \overline{U_{i,j}}$.  On the other
  hand, $1 \in N$, so $V_{i,j} \subseteq U_{i,j}$.  Therefore, if $j =
  \eta(i)$, then $a_\eta \in V_{i,j} \subseteq U_{i,j}$.  Putting
  everything together, we get
  \begin{gather*}
    j = \eta(i) \implies a_\eta \in U_{i,j} \\
    j \ne \eta(i) \implies a_\eta \notin \overline{U_{i,j}}.
  \end{gather*}
  The sets $U_{i,j}$ are uniformly interpretable, so we can find some
  formula $\varphi(x;y)$ such that each $U_{i,j}$ has the form
  $\varphi(\Mm;b_{i,j})$ for some $b_{i,j}$ (not the original ones).
  Then $(\varphi(\Mm;b_{i,j}) : i < r, ~ j < \omega)$ is the desired
  ict pattern.
\end{proof}

\section{Extensions and \textit{dfg}} \label{exts2}

Work in a highly resplendent monster model $\Mm$.  $\acl(-)$ always
means $\acl^\eq$.  All sets and parameters can come from $\Mm^\eq$ by
default.  ``Definable'' means ``interpretable.''
\begin{definition}
  A definable set $D$ is \emph{almost $A$-definable} if it is
  $\acl(A)$-definable, or equivalently, $\{\sigma(D) : \sigma \in
  \Aut(\Mm/A)\}$ is finite.  A global definable type $p$ is
  \emph{almost $A$-definable} if it is $\acl(A)$-definable, or
  equivalently, $\{\sigma(p) : \sigma \in \Aut(\Mm/A)\}$ is small.
\end{definition}
The following is folklore; see \cite[Lemma~3.13]{johnson} for a proof.
\begin{fact} \label{left-tran}
  Suppose $b$ realizes $p | A$ for some almost
  $A$-definable global type $p$.  Suppose $c$ realizes $q |
  (Ab)$ for some almost $Ab$-definable global type $q$.  Then $c$
  realizes $r | A$ for some almost $A$-definable global
  type $r$.
\end{fact}

\begin{definition}
  Let $G$ be an $A$-definable group.  Say that $G$ has \textit{dfg}
  \emph{over $A$} if there is a global definable type $p$ on $G$ such
  that $p$ and all its left-translates are almost $A$-definable.
\end{definition}
\begin{lemma} \label{first}
  Let $G$ be a definable \textit{dfg} group and $S$ be a definable set
  with a regular right action of $G$.  Suppose everything is
  $A$-definable, and $G$ has \textit{dfg} over $A$.  Then there is a
  global type on $S$ that is almost $A$-definable.
\end{lemma}
\begin{proof}
  For $b \in S$, let $b \cdot p$ denote the pushforward of
  the $A$-definable type $p$ along the map $x \mapsto b \cdot x$ from
  $G$ to $S$.  Note that $b \cdot p$ is a definable type on $S$.

  The set $\mathfrak{S} = \{b \cdot p : b \in S\}$ is small, because
  it is $\{b_0 \cdot g \cdot p : g \in G\}$ for any fixed $b_0 \in S$.
  If $\sigma \in \Aut(\Mm/\acl(A))$, then $\sigma$ fixes $p$ and
  $\sigma$ fixes $\mathfrak{S}$ setwise, since $\mathfrak{S}$ was
  defined in an invariant way.  Therefore any $b \cdot p$ has small
  orbit under $\Aut(\Mm/\acl(A))$, implying that $b \cdot p$ is
  almost $A$-definable.
\end{proof}
If $G$ is a $\varnothing$-definable group, let $\Mm \ltimes G$ be the new
structure obtained by adding a copy of $G$ as a new sort $S$, and
putting no structure on $S$ other than the regular right action of
$G$.  For any $g \in G$, there is an automorphism of $\Mm \ltimes G$
fixing $\Mm$ and acting as left translation by $g$ on the new sort
$S$.  In fact, $\Aut(\Mm \ltimes G) \cong \Aut(\Mm) \ltimes G$.

This construction is called ``Construction $C$'' in \cite[\S
  1]{udi-anand}, where it is attributed to Hrushovski's thesis.  It
also appears in \cite{NIPguide} above Lemma~8.19.  As mentioned in
\cite{NIPguide}, $\Mm \ltimes G$ is a conservative extension of $\Mm$,
in the sense that it introduces no new $\varnothing$-definable or definable sets
on $\Mm$.  After naming the element $1 \in S$, the two structures are
bi-interpretable.  Since we assumed $\Mm$ was very resplendent, $\Mm
\ltimes G$ will be too.
\begin{lemma} \label{second}
  Let $A \subseteq \Mm$ be a small set of parameters.  Suppose that in
  $\Mm \ltimes G$, there is a global type $p$ on $S$ that is almost
  $A$-definable.  Then $G$ has \textit{dfg} over $A$.
\end{lemma}
\begin{proof}
  For $b, s \in S$, let $b^{-1} \cdot s$ denote the unique $x \in G$
  such that $s = b \cdot x$.  Let $b^{-1} \cdot p$ denote the
  pushforward of $p$ along the map $x \mapsto b^{-1} \cdot x$ from $S$
  to $G$.  Then $b^{-1} \cdot p$ is a definable type on $G$.  If
  $\sigma \in \Aut(\Mm/A)$, we can extend $\sigma$ to $\hat{\sigma}
  \in \Aut((\Mm \ltimes G) / A)$ fixing $b$.  Then
  \begin{equation*}
    \sigma(b^{-1} \cdot p) = \hat{\sigma}(b^{-1} \cdot p) = b^{-1} \cdot \hat{\sigma}(p).
  \end{equation*}
  There are only a small number of possibilities for
  $\hat{\sigma}(p)$, and so $b^{-1} \cdot p =: q$ is almost
  $A$-definable.

  If $g \in G$, then $g \cdot b^{-1} \cdot x = (b \cdot g^{-1})^{-1}
  \cdot x$ for $x \in S$, and so
  \begin{equation*}
    g \cdot q = g \cdot b^{-1} \cdot p = (b \cdot g^{-1})^{-1} \cdot p = (b')^{-1} \cdot p
  \end{equation*}
  for $b' = b \cdot g^{-1}$.  Replacing $b$ with $b'$ in the argument
  above, we see that $(b')^{-1} \cdot p = g \cdot q$ is almost
  $A$-definable.  In other words, every translate $g \cdot q$ of $q$
  is almost $A$-definable, showing $G$ has \textit{dfg} over $A$.
\end{proof}
Lemmas~\ref{first} and \ref{second} are formally analogous to
\cite[Lemma~8.19]{NIPguide}, replacing ``non-forking over $A$'' with
``almost $A$-definable.''
\begin{theorem} \label{dfg-ext-2}
  If $1 \to N \to G \to H \to 1$ is a short exact sequence of
  definable groups, and $N, H$ have \textit{dfg}, then $G$ has
  \textit{dfg}.
\end{theorem}
\begin{proof}
  Naming parameters, we may assume the whole sequence is $\varnothing$-definable,
  and that $N$ and $H$ have \textit{dfg} over $\varnothing$.
  Construct $\Mm \ltimes G$.  Let $S$ be the new sort with a regular
  right action of $G$.  Let $S'$ be the quotient $S/N$.  Then $S'$ has
  a regular right action by $H$.  By Lemma~\ref{first}, there is an
  almost $\varnothing$-definable global type $p$ on $S'$.  Take $b$
  realizing $p | \varnothing$.  Let $S''$ be the fiber of
  $S \to S'$ over $b \in S'$.  Then $S''$ is a $b$-definable set with
  a $b$-definable regular right action by $N$.  By Lemma~\ref{first},
  there is an almost $b$-definable global type $q$ on $S''$.  Let $c$
  realize $q | b$.  Note $c \in S$.  By
  Fact~\ref{left-tran}, there is an almost
  $\varnothing$-definable global type $r$ on $S$ such that $c$
  realizes $r | \varnothing$.  By Lemma~\ref{second}, $G$
  has \textit{dfg}.
\end{proof}
Theorem~\ref{dfg-ext-2} generalizes one direction of
Lemma~\ref{dfg-ext}.  We cannot expect the reverse direction to hold
(if $G$ has \textit{dfg}, then $N$ and $H$ have \textit{dfg}).  For
example, in $\pCF^\eq$, the short exact sequence
\begin{equation*}
  0 \to \Zz_p \to \Qq_p \to \Qq_p/\Zz_p \to 0
\end{equation*}
is a counterexample: $\Qq_p$ has \textit{dfg} but $\Zz_p$ does not.
So the use of definable Skolem functions in Lemma~\ref{dfg-ext} is
essential.

\bibliographystyle{alpha} \bibliography{references.bib}{}

\end{document}